\documentclass[12pt]{amsart}
\usepackage{amsmath,amsthm,amsfonts,amssymb,latexsym,enumerate} 
\newcommand{\QQ}{{\mathbb{Q}}}
\newcommand{\PP}{{\mathcal{P}}}

\newcommand{\bO} {\mathbf O}

\newcommand{\bC} {\mathbf C}
\newcommand{\bZ}{{\mathbf Z}}
\newcommand{\bF}{{\mathbf F}}

\newcommand{\aS} {\mathcal S}
\newcommand{\Aut}{{{\operatorname{Aut}}}}

\newcommand{\Irr}{{{\operatorname{Irr}}}}

\newcommand{\SL}{\operatorname{SL}}
\newcommand{\cd}{\operatorname{cd}}

\newcommand{\Soc}{\operatorname{Soc}}
\newcommand{\Ker}{\operatorname{Ker}}
\newcommand{\Syl}{\operatorname{Syl}}
\newcommand{\Sym}{\operatorname{Sym}}
\newcommand{\dl}{\operatorname{dl}}
\newcommand{\PSL}{\operatorname{PSL}}
\newcommand{\PSp}{\operatorname{PSp}}
\newcommand{\Alt}{\operatorname{Alt}}

\newtheorem{thm}{Theorem}[section]
\newtheorem{lem}[thm]{Lemma}
\newtheorem{con}[thm]{Conjecture}

\newtheorem{cor}[thm]{Corollary}
\newtheorem{que}[thm]{Question}

\newtheorem*{queA}{Question A}
\newtheorem*{conB}{Conjecture B}
\newtheorem*{conC}{Conjecture C}

\theoremstyle{definition}

\numberwithin{equation}{section}

\begin{document}
%%%%%%%%%%%%%%%%%%%%%%%%%%%%%%%%%%%%%%%%%%%%%%%%%%%%%%%%%%%%%
\title[Character degrees of finite groups]{Methods and questions in character degrees of finite groups}
%%%%%%%%%%%%%%%%%%%%%%%%%%%%%%%%%%%%%%%%%%%%%%%%%%%%%%%%%%%%%

\author{Alexander Moret\'o}
\address{Departamento de Matem\'aticas, Universidad de Valencia, 46100
  Burjassot, Valencia, Spain}
\email{alexander.moreto@uv.es}

\thanks{I am very grateful Noelia Rizo  for many inspiring conversations and for encouraging me to write this paper. I also thank Nguyen Ngoc Hung,  Gabriel Navarro and Benjamin Sambale for helpful comments on a previous version of this paper.  Last, but not least, I thank the anonymous  referees for their very careful reading of the paper and many helpful suggestions. The research of the author is supported by  Ministerio de Ciencia e Innovaci\'on (Grant PID2019-103854GB-I00 funded by MCIN/AEI/ 10.13039/501100011033).}



\date{\today}

\begin{abstract}
We present some variations on some of the main open problems on character degrees. 
We collect some of the methods that have proven to be very useful to work on these problems. These methods are also useful to solve certain problems on zeros of characters, character kernels and fields of values of characters.
\end{abstract}

\maketitle

%\pagestyle{myheadings}
%\markboth{for personal use only}{preliminary}

\centerline{\it
To Pham Huu Tiep, on his $60$th birthday}

%%%%%%%%%%%%%%%%%%%%%%%%%%%%%%%%%%%%%%%%%%%%%%%%%%%%%%%%%%%%%%%%%%%%%%%%%
\section{Introduction}  

In this mostly expository paper, we collect some of the ideas that have proven to be very useful in our research on character theory of finite groups. 
Many of the problems are divided in three layers. First, the problem is studied in the case of $p$-groups, then solvable groups and finally arbitrary finite groups.

As is well-known, given the vast diversity of $p$-groups, it is not easy to develop a general theory of $p$-groups. Much less, a general character theory of $p$-groups. For this reason, we cannot single out some ideas that will be useful to solve many problems on characters of $p$-groups. In our experience, each problem has required its own ad hoc techniques. For this reason, when we discuss $p$-groups in this article, we will restrict ourselves to present some of the problems we have worked on, some open questions, and the interplay between these problems on characters of $p$-groups with problems on wider classes of groups. It is our hope that, some day, we will see a more developed character theory of $p$-groups. 

Regarding solvable groups, we will focus on the large orbit theorem proved by T. Wolf and the author in 2004 (Theorem E of \cite{mw}). There have been a number of applications of this theorem. These applications concern not only problems on  characters of solvable groups, but also problems on conjugacy class sizes of solvable groups. It partially explains the duality that has been observed between results on conjugacy class sizes and on character degrees.
We expect more applications to be found in the future.  

Probably the main open problems on character degrees of solvable groups are the Isaacs-Seitz conjecture, Huppert's $\rho$-$\sigma$ conjecture and Gluck's conjecture (see Sections 4 and 5 for details). As we will see, Theorem E of \cite{mw} is relevant both for Huppert's $\rho$-$\sigma$ conjecture and Gluck's conjecture but, unfortunately, not much progress has been made on any of these conjectures in the last couple of decades.  One of the main goals of this paper is to present stronger versions of these conjectures, with the hope that they will either allow for an alternative approach to the problem or to gain some insight. The following is our strong form of the Isaacs-Seitz conjecture. Given a group $G$, $\cd(G)$ is the set of degrees of the irreducible characters of $G$.

\begin{queA}
Let $G$ be a non-monomial solvable group. Is it true that $|\cd(G')|<|\cd(G)|$?
\end{queA}

The Isaacs-Seitz conjecture is known for monomial groups by Taketa's
argument. For non-monomial groups, an affirmative answer to Question A indeed
implies the Isaacs-Seitz conjecture.
On the other hand, 
it is not that easy to find non-abelian monomial counterexamples to the conclusion in Question A.  Essentially, the only such counterexample we are aware of is $G={\tt PrimitiveGroup(81,64)}$ (and direct products of this group with abelian groups). This example was found by B. Sambale, to whom we thank.  Note that $|G|=2^9\cdot3^4$ and $G$ has a normal abelian Sylow $3$-subgroup. It follows from Theorem 6.22 and Theorem 6.23 of \cite{isa} that $G$ is monomial. 

Huppert's $\rho$-$\sigma$ conjecture also has a version for arbitrary finite groups.  As usual, if $G$ is a finite group $\rho(G)$ is the set of primes that divide the degree of some irreducible characters of $G$. The following is inspired by the proofs of some of the results on this conjecture.

\begin{conB} 
Let $G$ be a finite group. Then there exist $\chi_1,\chi_2,\chi_3\in\Irr(G)$ such that  all the primes in $\rho(G)$ divide $\chi_1(1)\chi_2(1)\chi_3(1)$. Furthermore, if $G$ is solvable, then two characters are enough.
\end{conB}

A version of Gluck's conjecture for arbitrary finite groups has recently been proposed in \cite{chmn}.  Motivated by the work in \cite{mw} and \cite{m06} (see for instance  Question 2.2 of \cite{mw} and Section 2 of \cite{m06}) we propose the following. We write $\bF(G)$ to denote the Fitting subgroup of the group $G$.

\begin{conC}
Let $G$ be a finite group.  Then there exist $\chi_1,\chi_2,\chi_3,\chi_4\in\Irr(G)$ such that  $|G:\bF(G)|$ divides $\chi_1(1)\chi_2(1)\chi_3(1)\chi_4(1)$. Furthermore, if $G$ is solvable, then two characters are enough.
\end{conC}

The aforementioned generalization of Gluck's conjecture asserts that $|G:\bF(G)|\leq b(G)^3$, where $b(G)$ is the largest degree of the complex irreducible characters of $G$. Unfortunately, $3$ characters are not enough in Conjecture C, but the only counterexamples among the simple groups in the Atlas \cite{atl} are $\Alt(7)$, $\Alt(13)$ and $M_{22}$.

In the case of arbitrary finite groups there are two basic but fundamental results that have been routinely used. On the one hand, if $S$ is a nonabelian simple group, then there exists  a nonprincipal  irreducible character $\varphi\in\Irr(S)$ that extends to $\Aut(S)$. This is Lemma 4.2 of \cite{m07} (see also Lemma 2.11 of \cite{m05}). As a sign of the usefulness of this result, we remark that it was already used a couple of times before it appeared in print (see \cite{ms1} and \cite{mm04}).

The second result we are referring to is  that if $A$ is a permutation group on a finite set and $\Gamma=G\wr A$ then any $\Gamma$-invariant irreducible character of the base group extends to $\Gamma$. This follows from  Lemma 1.3 of \cite{mat}.  The result, as stated,  seemed to be folklore before 1995 (see for instance Section 4.3 of \cite{jk}) but Mattarei's remarkable Lemma 1.3 gives additional information on the values of the extended character that as we will see is sometimes useful. We will therefore refer to this result as Mattarei's lemma.

We will illustrate the relevance of both of these results in many problems. For instance, they have been used to obtain generalizations on the two most celebrated results on character degrees, the It\^o-Michler Theorem and Thompson's Theorem (see \cite{hun, ht, nav2} for instance). They were first used in conjunction in \cite{m07} in relation to Brauer's Problem 1 in his famous list of open problems in representation theory of finite groups  \cite{bra}.

Our exposition is far from being exhaustive. We have limited ourselves to explain the ideas that we have used most often in our research. Our goal is to show a pattern that could, hopefully, be used by others in problems on characters of finite groups.
There is one exception to this rule, however. There is a large and growing number of results that guarantee the existence of irreducible characters of simple groups with certain properties that extend irreducibly to (large subgroups of) the automorphism group. Given that it is not easy to keep track of what is known and what not, we have collected in Section 6 most of the known results on this topic.

 The problems we consider here concern mostly character degrees. We have omitted references to results on conjugacy class sizes, but these results appear in several of the references mentioned here. In the last section, we explain how the techniques presented  are also useful for problems on zeros of characters, fields of values and character kernels. The problems and results discussed on solvable groups complement those in the recent expository paper \cite{nsur}
 
\section{The general picture}

Recall that if $G$ is a group, then $\bF(G)$ is the largest normal nilpotent subgroup. In particular, it is the direct product of its Sylow subgroups. This is one of the ways $p$-groups come into play in problems on solvable or arbitrary groups. The following key result allows to translate many problems on character degrees of solvable groups to problems on orbit sizes when a solvable group acts on a finite module. Recall also that the Frattini subgroup $\Phi(G)$ is the intersection of the maximal subgroups of $G$ and $\Phi(G)\leq \bF(G)$. 

\begin{thm}[Gasch\"utz] 
\label{gas}
Let $G$ be a solvable group. Then $\bF(G/\Phi(G))=\bF(G)/\Phi(G)$  is a faithful and completely reducible $G/\bF(G)$-module (possibly of mixed characteristic). Furthermore, $G/\Phi(G)$ splits over $\bF(G)/\Phi(G)$. 
\end{thm}

\begin{proof}
See III.4.2, III.4.4 and III.4.5 of \cite{hup}.
\end{proof}

The next key subgroup in our approach when $G$ is an arbitrary group is the solvable radical $R(G)$, which is the largest normal solvable subgroup. We often work in $G/R(G)$, which is a group with trivial solvable radical. 

Note that if $G$ is a group with trivial solvable radical, then the generalized Fitting subgroup of $G$, $\bF^*(G)$, coincides with the socle of $G$ and is the direct product of all the minimal normal subgroups of $G$.  Since $\bC_G(\bF^*(G))\leq \bF^*(G)$ (by Theorem 6.5.8 of \cite{ks}) and $\bC_{\bF^*(G)}(\bF^*(G))=\bZ(\bF^*(G))=1$ in this case, we deduce that if $G$ is a group with trivial solvable radical then $\bC_G(\bF^*(G))=1$, so $G$ is isomorphic to a subgroup of $\Aut(\bF^*(G))$. 

Now, write $\bF^*(G)=N_1\times\cdots\times N_t$, where $N_i$ is a direct product of, say,  $n_i$ copies of a nonabelian simple group $S_i$ and $S_i\not\cong S_j$ for $i\neq j$. As pointed out in Lemma 2.3 of \cite{m05}, for instance, $\Aut(\bF^*(G))\cong\Aut(N_1)\times\cdots\times\Aut(N_t)$. Notice also that $\Aut(N_i)\cong\Aut(S_i)\wr\Sym(n_i)$. Thus, if we write $\Gamma=\Aut(\bF^*(G))$ then $\Gamma=HN$ where $N=\Aut(S_1)^{n_1}\times\cdots\times\Aut(S_t)^{n_t}\trianglelefteq\Gamma$, $H=\Sym(n_1)\times\cdots\times\Sym(n_t)$ and $H\cap N=1$.

Put $K=G\cap N$. Note that $K\trianglelefteq G$ and $\bF^*(G)=S_1^{n_1}\times\cdots\times S_t^{n_t}\leq K\leq\Aut(S_1)^{n_1}\times\cdots\times\Aut(S_t)^{n_t}=N$. On the other hand $G/K$, which is isomorphic to a subgroup of $H$,  acts on $\bF^*(G)$ by permuting the simple direct factors. This is the reason for the relevance of  Mattarei's lemma. 

\begin{thm}
Let $G$ be a group with trivial solvable radical. Write $\bF^*(G)=N_1\times\cdots\times N_t$, where $N_i$ is a direct product of, say,  $n_i$ copies of a nonabelian simple group $S_i$ and $S_i\not\cong S_j$ for $i\neq j$. Put $\Gamma=HN$ where $N=\Aut(S_1)^{n_1}\times\cdots\times\Aut(S_t)^{n_t}\trianglelefteq\Gamma$ and $H=\Sym(n_1)\times\cdots\times\Sym(n_t)$ acts on $N$ by permuting the copies of $\Aut(S_i)$. Then $G\leq \Gamma$. Furthermore, if $K=G\cap N$ then $G/K$ acts on $\bF^*(G)$ by permuting the simple direct factors.
\end{thm}

Sometimes it is more convenient to work in $G/\bC_G(N_i)\leq\Aut(S_i)\wr\Sym(n_i)$ (see for instance the proof of Theorem C of \cite{m07}, where these ideas were used for the first time).

As we will see, in many problems we will consider first the case of $p$-groups. These results will be applied to the Fitting subgroup of an arbitrary finite group. Next, we consider the case of solvable groups, and apply these results to the solvable radical $R(G)$ of an arbitrary finite groups. Finally, we consider the general case. This implies studying groups with trivial solvable radical, which can be decomposed as above, and then apply these results to the quotient $G/R(G)$ of an arbitrary finite group $G$.

\section{$p$-groups}

Let $P$ be a nonabelian $p$-group. Write $|P:\bZ(P)|=p^{2n+e}$ for some integer $e\in\{0,1\}$. Let $p^m$ be the smallest degree of the nonlinear irreducible characters of $P$ and put $m(P)=m$. By Corollary 2.30 of \cite{isa}, $m(P)\leq n$. Groups with $m(P)=n$ were characterized in Theorems B and C of \cite{fm} in terms of the normal subgroups of $P$: 
they are the $p$-groups $P$ with the property that for any $N\trianglelefteq P$, $P'\leq N$ or $|N\bZ(P):\bZ(P)|\leq p$. 
 As mentioned in Section 3 of \cite{m03}, given any nonabelian $p$-group $P$, we may define $m_1(P)$ to be the integer such that $m(P)=n-m_1(P)$, so the results in \cite{fm} characterize groups with $m_1(P)=0$.
This above mentioned condition on the normal subgroups of a $p$-group was generalized by M. Isaacs \cite{isa02}. He defined $a(P)=a$ to be the smallest integer $a$ such that if $N$ is normal in $P$ and $|N\bZ(P):\bZ(P)|\geq p^a$, then $P'\leq N$. As suggested in Section 3 of \cite{m03}, there should be generalizations on the results in \cite{fm} on groups with $m_1(P)=0$ and $a(P)=0$ or $1$ to arbitrary values of $m_1(P)$ and $a(P)$. However, no progress has been made yet on this. 

		Recently, motivated by a problem on $p$-blocks of $p$-solvable groups, a dual problem has come up (see \cite{mn}). As pointed out in the proof of Lemma 2.1 of \cite{mn}, it is easy to see that if $P$ is a $p$-group and $A\leq P$ is abelian, then $\chi(1)$ divides $|P:A|$ for all $\chi\in\Irr(P)$. In particular, if $|P|=p^n$ and the exponent of $P$ is $p^e$, then $\chi(1)\leq p^{n-e}$ for every $\chi\in\Irr(P)$.  If $b(P)$ is the largest degree of the irreducible characters of $P$, then $b(P)=p^{n-e-s(P)}$ for some nonnegative integer $s(P)$. Notice that $s(P)=0$ if and only if $P$ has an irreducible character that is induced from some cyclic subgroup. These groups were characterized in Theorem B of \cite{mn} (when $p$ is odd). Some results on groups with $s(P)>0$ were obtained in Theorem D of \cite{mn}, which bounds the (minimal) number of generators of an odd $p$-group in terms of $s(P)$. For the block-theoretic purposes of \cite{mn},  these results are not strong enough when $p=3$.
		The following includes Questions 1 and 2 of \cite{mn}:
		
		\begin{que}
		Let $P$ be a $3$-group. Is it true that if $s(P)=1$ then the minimal number of generators of $P$ is $d(P)\leq 5$? Is it true that if $s(P)=2$ then $d(P)\leq 10$?
		\end{que}
		
		As discussed in Section 7 of \cite{mn}, an affirmative answer to these questions on $3$-groups, would allow to extend the block-theoretic Theorem A of \cite{mn} to $p=3$.
		
		We conclude this section by mentioning another problem we worked on twenty years ago: class-bounding sets. Let $\aS$ be a set of powers of a prime number $p$ containing $1$. Is there an absolute bound for the nilpotence class of a $p$-group $P$ with set of character degrees $\aS$? The sets $\aS$ for which such a bound exist are called class-bounding sets. It is known, for instance, that if $p\in\aS$, then $\aS$ is not class-bounding (see \cite{is}). On the other hand, if there exists an integer $a$ such that $\aS\subseteq\{1,p^a,\dots,p^{2a}\}$ then $\aS$ is class-bounding (this is Theorem A of \cite{im}). There are other results on this problem. We refer the reader to Section 2 of \cite{m03} and its references for more details. We are just as far as then from determining the class-bounding sets. 
		
		\section{From $p$-groups to $p$-solvable groups}
		
		We have already seen in the previous section an example of a problem on characters of $p$-solvable groups that is reduced to a problem on $p$-groups. In this section, we present two more problems on solvable (or $p$-solvable) groups in which $p$-groups play a very important role.

		One of the main problems on character degrees of solvable groups is the Isaacs-Seitz conjecture.  The Isaacs-Seitz conjecture asserts that the derived length of $G$ is $\dl(G)\leq|\cd(G)|$. (The origin of this conjecture is not clear. The first general bound on this problem is \cite{isa75}, where it is mentioned that this inequality is conjectured. The case when $|\cd(G)|\leq3$ had been proved before in \cite{ip, isa68}. According to p. 109 of \cite{nav10}, Isaacs attributed this problem to Seitz but it seems that Seitz was not fully aware  of having formulated the problem.) We refer the reader to Section 16 of \cite{mawo} for an account of the main results known on this around half a century old problem. Despite the fact that \cite{mawo} was published in 1993, there is not much progress on this problem to report on. The most remarkable results since then are due to T. Keller and do not directly aim at proving this conjecture, but an asymptotically stronger bound (a logarithmic bound). We again refer the reader to Section 2 of \cite{m03} for more details. We remark that in the case of $p$-groups, the Isaacs-Seitz conjecture follows from Taketa's argument  (Theorem 5.12 of \cite{isa}). However, even in this case we cannot report on much progress on the problem of finding a logarithmic bound for the derived length of a $p$-group $P$  in terms of $|\cd(P)|$ (Conjecture 2 of \cite{m03}). The only results that we are aware on this concern a special type of $p$-groups of maximal class (see \cite{krt, bw2}). 
	
		As mentioned in the Introduction, it seems that very often the number of character degrees of a non-abelian solvable group $G$ exceeds	the number of character degrees of the derived subgroup $G'$. 	Even if Question A has a negative answer, it would be interesting to understand the counterexamples.
		
		\begin{que}
		Is it possible to describe the solvable groups $G$ with the property that $|\cd(G')|\geq|\cd(G)|$?
		\end{que}

		The second problem we will consider in this section looks at relating the structure of a $p$-solvable group with the set of character degrees of a Sylow $p$-subgroup. It was proved in Theorem 5.2 of \cite{em} that if $p$ is an odd prime, $G$ is a solvable group, $P\in\Syl_p(G)$ and $\cd(P)\subseteq\{1,p^a,\dots,p^{2a}\}$ for some $a>1$ then the $p$-length of $G$ is at most $1$. In view of the formal similarity between this result and Theorem A of \cite{im} (which has already appeared in Section 2), we said in \cite{em} that a set $\aS$ of powers of $p$ containing $1$ bounds the $p$-length if whenever $G$ is a $p$-solvable group, $P\in\Syl_p(G)$ and $\cd(P)=\aS$, then the $p$-length of $G$ is at most $1$. The following is Question 5.3 of \cite{em}.

\begin{que}
Given a prime $p$, which are the sets of powers of $p$ that bound the $p$-length?
\end{que}

It would be interesting to confirm (or deny) the existence of strong relationships between the class-bounding sets and the sets that bound the $p$-length. 

\section{Solvable groups: orbits on modules}

Gasch\"utz's Theorem \ref{gas} allows to reduce the study of many problems on solvable groups to that of the study of orbits when a solvable group acts on a finite module. In this section, we will restrict ourselves to pointing out some of the applications of the following large orbit theorem. Given an integer $i$, $\bF_i(G)$ stands for the $i$th term of the ascending Fitting series of $G$. The Fitting height $h(G)$ of a group $G$ is the smallest integer $i$ such that $G=\bF_i(G)$.

\begin{thm}[Theorem E of \cite{mw}, Theorem 3.4 of \cite{yan09}]
\label{lo}
Let $V$ be a finite faithful completely reducible $G$-module (possibly of mixed characteristic), where $G$ is a solvable group. Then
\begin{enumerate}
\item
There exists $v\in V$ such that $\bC_G(v)\leq \bF_7(G)$.
\item
If $|GV|$ is odd then there exists $v\in V$ in a regular orbit of $\bF(G)$ such that $\bC_G(v)\leq \bF_2(G)$.
\end{enumerate}
\end{thm}

Theorem E of \cite{mw} is exactly this statement with $\bF_9(G)$ instead of $\bF_7(G)$ in part (i). The improvement to $\bF_7(G)$ is Theorem 3.4 of \cite{yan09}. It is remarkable that is was proved in \cite{yan09} that $\bF_7(G)$ is the smallest term in the ascending Fitting series for which a statement like (i) holds. Essentially, this theorem allows to reduce many problems on arbitrary solvable groups to solvable groups of bounded (at most $8$) Fitting height. These problems concern mainly character degrees, zeros of characters and conjugacy class sizes. 

We start with applications to Huppert's $\rho$-$\sigma$ conjecture. Given a finite group $G$, $\rho(G)$ is the set of primes that divide the degree of some irreducible character of $G$. Given an integer $n$, $\pi(n)$ is the set of prime divisors of $n$ and $\sigma(G)=\max\{|\pi(\chi(1))|\mid\chi\in\Irr(G)\}$. Huppert's $\rho$-$\sigma$ conjecture asks the following (see p. 220 and Remark 17.9(a) of \cite{mawo}):
\begin{enumerate}
\item
Is there a function $f$ (independent of $G$) such that $|\rho(G)|\leq f(\sigma(G))$ for every finite group $G$?
\item
Is it true that $|\rho(G)|\leq2\sigma(G)$ if $G$ is solvable?
\end{enumerate}

We will discuss part (i) in Section 7. Now, we show that Theorem \ref{lo} implies the existence of a linear bound when $G$ is solvable.

\begin{thm}
Let $G$ be a solvable group. Then there exist $\chi_1,\dots,\chi_8\in\Irr(G)$ such that $\rho(G)=\pi(\chi_1\cdots\chi_8(1))$. In particular, $|\rho(G)|\leq8\sigma(G)$.
\end{thm}

\begin{proof}
		It follows from Theorem \ref{lo}(i) and Gasch\"utz's theorem that there exists $\mu\in\Irr(\bF_8(G))$ such that $\chi_8=\mu^G\in\Irr(G)$ (see Theorem 5.2 of \cite{yan09}). By Proposition 17.3 of \cite{mawo}, for every $i=2,\dots,7$ there exists $\varphi_i\in\Irr(\bF_i(G)/\bF_{i-2}(G))$ such that $\varphi_i(1)$ is divisible by all the primes in $|\bF_i(G):\bF_{i-1}(G)|$. Now, it follows from Clifford's theory that if $\chi_i\in\Irr(G)$ lies over $\varphi_i$ then $\pi(|\bF_i(G):\bF_{i-1}(G)|)\subseteq\pi(\chi_i(1))$. Finally, we take $\varphi_1\in\Irr(\bF(G))$ such that $\varphi_1(1)$ is divisible by all the primes $p$ such that $\bO_p(G)$ is not abelian. Let $\chi_1\in\Irr(G)$ lying over $\varphi_1$. It follows that $\rho(G)=\pi(\chi_1\cdots\chi_8(1))$, as desired.
		\end{proof}
		
		As with the Isaacs-Seitz conjecture, not much progress has been made on part (ii) of Huppert's conjecture  since the publication of \cite{mawo}. We refer the reader to Section 17 of \cite{mawo} for a summary of the main results on this topic. Until very recently,   the best known bound was  $|\rho(G)|\leq3\sigma(G)+2.$  This bound already appears in \cite{mawo}. Recently, this has been slightly improved to $|\rho(G)|\leq3\sigma(G)$ in \cite{ly} and in \cite{adp}. On the other hand, notice that while Huppert's conjecture just asks for the existence of an irreducible character of degree divisible by ``many" primes, the first part of our theorem provides the existence of a few characters whose degrees cover $\rho(G)$. We propose the following strong form of Huppert's $\rho$-$\sigma$ conjecture.
		
		\begin{que}
		\label{hupsol}
		Let $G$ be a solvable group. Do there exist $\chi_1,\chi_2\in\Irr(G)$ such that  $\rho(G)=\pi(\chi_1\chi_2(1))$?
		\end{que}
		
		Also, we note that for odd order groups Theorem \ref{lo} provides the best known bounds in the $\rho$-$\sigma$ problem, both for character degrees and for conjugacy class sizes. We refer the reader to Section 2.2 of \cite{mw}. 
		
		A number of variations of the $\rho$-$\sigma$ conjecture have been studied in the literature. We think that the next variation would help us to get a better understanding of  the set of character degrees of finite groups. As usual, $\pi$ is a set of primes.  We will write $\rho_{\pi}(G)$ to denote the set of primes that divide the degree of some irreducible character of $G$ of $\pi$-degree. Put $\sigma_{\pi}(G)=\max\{|\pi(\chi(1))|\mid\chi\in\Irr_{\pi}(G)\}$, where $\Irr_{\pi}(G)$ is the set of irreducible characters of $G$ of $\pi$-degree. 
		
		\begin{que}
		Let $\pi$ be a set of primes. 
		\begin{enumerate}
		\item
		Is it true that there exists a real-valued function $f$ such that $|\rho_{\pi}(G)|\leq f(\sigma_{\pi}(G))$? Is it true that even $|\rho_{\pi}(G)|\leq3\sigma_{\pi}(G)$?
		\item
		Is it true that if $G$ is solvable then $|\rho_{\pi}(G)|\leq 2\sigma_{\pi}(G)$?
		\end{enumerate}
		\end{que}
		
		Note that as an immediate consequence of the It\^o-Michler theorem, we have that if $G$ is any finite group, $\rho(G)$ is the set of prime divisors $p$ of $|G|$ such that $G$ does not have a normal abelian Sylow $p$-subgroup. It does not seem easy however to get a group-theoretic characterization of $\rho_{\pi}(G)$. As a consequence of Theorem A of \cite{nrs}, we can get such a characterization for solvable groups (although a rather technical one).

		Another longstanding conjecture on solvable groups is the Isaacs-Navarro-Wolf conjecture on zeros of characters. Given an element $x\in G$, we say that $x$ is nonvanishing if $\chi(x)\neq0$ for every $\chi\in\Irr(G)$. It was conjectured in \cite{inw} that if $G$ is solvable and $x\in G$ is nonvanishing, then $x\in \bF(G)$. As we have just seen in the previous proof, if $G$ is solvable there exists $\chi\in\Irr(G)$ that is induced from $\bF_8(G)$. We thus have:
		
		\begin{thm}
		Let $G$ be a nonsolvable group and $x\in G$ be a nonvanishing element. Then $x\in \bF_8(G)$.
		\end{thm}

		Recall that Gluck's conjecture asserts that if $G$ is a solvable group then $|G:\bF(G)|\leq b(G)^2$, where $b(G)$ is the largest degree of the irreducible characters of $G$. 
		The following is a simultaneous generalization of Gluck's conjecture (probably the remaining among the main open problems on character degrees of solvable groups that has not been mentioned yet) and the Isaacs-Navarro-Wolf conjecture appears as Question 2.2 of \cite{mw}
		
		\begin{que}
		\label{glusol}
		Let $G$ be a solvable group. Do there exist $\chi_1,\chi_2\in\Irr(G)$ such that $|G:\bF(G)|$ divides $\chi_1(1)\chi_2(1)$ and $\chi_1\chi_2$ vanishes on $G-\bF(G)$? 
		\end{que}
		
		There are many other applications of Theorem \ref{lo}. We refer the reader to \cite{mw} and to the papers that quote \cite{mw}.	
			
		\section{Simple groups and their automorphisms: extendible characters}

		Nowadays, the relevance of the following result is well-known.
		
		\begin{thm}
		\label{extsim}
		Let $S$ be a nonabelian simple group. Then there exists a nonprincipal $\varphi\in\Irr(S)$ that extends to $\Aut(S)$. 
		\end{thm}
		
		This appeared first as Lemma 2.11 of \cite{m05}. A proof appeared in Lemma 4.2 of \cite{m07}. This result has been used multiple times. Sometimes, (partially) stronger versions of this result have been needed. Our goal in this section is to collect these stronger versions. 
		
		Often we want the degree of the extending character to be large. In this context, probably the most remarkable result is the following, which improves on an earlier result in \cite{chmn}. 
		
		\begin{thm}[Theorem 2.1 of \cite{hls}]
		Let $S$ be a nonabelian simple group. Suppose that $S\not\cong\PSL_2(q)$. Then there exists $\varphi\in\Irr(S)$ that extends to $\Aut(S)$ and $\varphi(1)\geq|S|^{3/8}$.
		If $S\cong\PSL_2(q)$ then there exists $\varphi\in\Irr(S)$ that extends to $\Aut(S)$ and $\varphi(1)\geq|S|^{1/3}$.
		\end{thm}

On other occasions, we need more than one character that extends and we want the degrees to be different (or even coprime).

		\begin{thm}[Theorems 3 and 4 of \cite{bclp}]
		\label{bclp}
	Let $S$ be a nonabelian simple group that is not of Lie type. Then there exist nonprincipal $\varphi_1,\varphi_2\in\Irr(S)$ with $(\varphi_1(1),\varphi_2(1))=1$ that extend to $\Aut(S)$.
		\end{thm}
		
		For groups of Lie type, we have the following variation of Theorem \ref{bclp}.
		
		\begin{thm}[Proposition 3.7 of \cite{mm07}]
		Let $S$ be a nonabelian simple group. Assume that $S\not\cong\PSL_2(3^f)$.  Then there exist nonprincipal $\varphi_1,\varphi_2\in\Irr(S)$ with $\varphi_1(1)\neq\varphi_2(1)$ that extend to $\Aut(S)$.
		\end{thm}
		
		Later, it was proved in \cite{cn} that if for any prime power $q$, $S\neq\PSL_2(q)$ then we can choose $\varphi_1$ and $\varphi_2$ in such a way that $\varphi_1(1)/\varphi_2(1)>|S|^{1/14}$. 
		
		It is easy to see that in the groups $\PSL_2(3^f)$ the Steinberg character is the unique nonprincipal irreducible character that extends to $\Aut(\PSL_2(3^f))$. Nevertheless, we have the following for arbitrary simple groups.
		
		\begin{thm}[Corollary 3.8 of \cite{mm07}]
		Let $S$ be a nonabelian simple group.   Then there exist nonprincipal $\varphi_1,\varphi_2\in\Irr(S)$ of different degrees such that $\varphi_1$ extends to $\Aut(S)$ and $\varphi_2$ extends to its inertia subgroup in $\Aut(S)$.
		\end{thm}
		
		A version of these results in \cite{mm07} for $p'$-degree ($p>3$) irreducible characters has been recently obtained in \cite{grs}.
		
		\begin{thm}[Theorem C of \cite{grs}]
		Let $S$ be a nonabelian simple group and let $p>3$ be a prime divisor of $|S|$.  Assume that $S$ is not one of $A_5$, $A_6$ for $p=5$ or one of $\PSL_2(q)$, $\PSL_3^{\varepsilon}(q)$, $\PSp_4(q)$ or $^2B_2(q)$. Then there exist nonprincipal $\varphi_1,\varphi_2\in\Irr_{p'}(S)$ of different degrees such that $\varphi_1$  and $\varphi_2$ extend to $\Aut(S)$.
		\end{thm}
		
		As a consequence, we have the following result without exceptions.
		
		\begin{thm}[Theorem B of \cite{grs}]
		Let $S$ be a nonabelian simple group and let $p>3$ be a prime divisor of $|S|$.  Then there exist nonprincipal $\varphi_1,\varphi_2\in\Irr_{p'}(S)$ of different degrees such that $\varphi_1$   extends $\Aut(S)$, $\varphi_2(1)$ does not divide $\varphi_1(1)$ and $\varphi_2$ is $P$-invariant for every $p$-subgroup $P\leq\Aut(S)$.
		\end{thm}
		
		For other related results, involving blocks, see Theorem C of \cite{rsv} and Proposition 2.1 of \cite{grss}. 
		The following result, essentially due to G. Navarro and P. H. Tiep \cite{nt}, appears as Lemma 2.2 of \cite{hun} and takes into account fields of values. If $N\trianglelefteq G$ and $\varphi\in\Irr(N)$ we write $I_G(\varphi)$ to denote the inertia subgroup of $\varphi$ in $G$. 
		
		\begin{thm}[Navarro-Tiep]
		Let $S$ be a nonabelian simple group and let $p$ be a prime divisor of $|S|$. Then there exists $\varphi\in\Irr_{p'}(S)$ such that $I_{\Aut(S)}(\varphi)$ has $p'$-index in $\Aut(S)$ and $\varphi$ extends to a $\QQ_p$-valued irreducible character of $I_{\Aut(S)}(\varphi)$.
		\end{thm}

		For characters of degree divisible by $p$, we have the following result.
		
		\begin{thm}[Theorem 3.1 of \cite{ht}]
		Let $S$ be a nonabelian simple group and let $p$ be a prime divisor of $|S|$.  Then there exists a nonprincipal $\varphi\in\Irr(S)$ of degree divisible by $p$ such that $\varphi$ extends to $I_{\Aut(S)}(\varphi)$.
		\end{thm}
		
		\begin{thm}[Theorem 3.2 of \cite{ht}]
		Let $p$ be a prime and let $S$ be a nonabelian simple $p'$-group.  Then there exists a nonprincipal $\varphi\in\Irr(S)$  such that $\varphi$ extends to $I_{\Aut(S)}(\varphi)$ and $p$ does not divide $|I_{\Aut(S)}(\varphi)|$.
		\end{thm}
		
	%	The following results take fields of values into account.
		
	%	\begin{thm}[Theorem 3.3 of \cite{nt}]
%		Let $p$ be a prime and let $S$ be a nonabelian simple group. Then there exists a nonprincipal $\varphi\in\Irr_{p'}(S)$ that extends to a $\QQ_p$-valued irreducible character of $I_{\Aut(S)}(\varphi)$ and such that $p$ does not divide $|\Aut(S):I_{\Aut(S)}(\varphi)|$.
%	\end{thm}

In the next result, we do not require any condition on the degree, but we get a rational extension.

\begin{thm}[Lemma 4.1 of \cite{hstv}]
\label{1}
   Let $S$ be a nonabelian simple group. Then there exists a nonprincipal $\varphi\in\Irr(S)$ that has an irreducible rational extension to $\Aut(S)$.   
 \end{thm}

Sometimes, we need many characters that extend. For alternating groups, we have the following. 

\begin{thm}[Lemma 2.4 of \cite{mproc}]
\label{2}
The number of irreducible characters of $A_m$ that extend to $\Sym(m)$ goes to infinity when $m\to\infty$.
\end{thm}

		For simple groups of Lie type, we have the following important result.
		
		\begin{thm}[Theorem 2.4 of \cite{mal08}]
		Let $S$ be a simple group of Lie type. Then all unipotent characters of $S$ extend to their inertia group in $\Aut(S)$.
		\end{thm}
		
		Furthermore, it was proved in Theorem 2.5 of \cite{mal08} that, with some exceptions listed in that theorem, all unipotent characters of groups of Lie type are $\Aut(S)$-invariant.

		\section{Extendibility of characters in nonsolvable groups}

		The results in the previous section, together with the result of Mattarei mentioned in the Introduction, allow to obtain results like the following.

		\begin{thm}
		\label{ext2}
		Let $G$ be a group with trivial solvable radical. Then there exists a faithful character $\beta\in\Irr(\bF^*(G))$ that extends irreducibly to $G$.
		\end{thm}
		
		\begin{proof}
		Notice that $\bF^*(G)$ is isomorphic to a direct product of copies of nonabelian simple groups. Write $\bF^*(G)=N_1\times\cdots\times N_t$, with $N_i=S_i\times\cdots\times S_i$ for every $i$ with $S_i$ simple and $S_i\neq S_j$ if $i\neq j$.  Let $n_i$ be the number of copies of $S_i$ that appear as direct factor of $N_i$. Thus $G$ is isomorphic to a subgroup of 
		$$
		\Aut(\bF^*(G))=\prod_i\Aut(N_i)=\prod_i\left(\Aut(S_i)\wr\Sym(n_i)\right)=\Gamma.
		$$
		We will see that $\bF^*(G)$ has some faithful character that extends to $\Gamma$.
		By Theorem \ref{extsim}, for every $i$ there exists $1_{S_i}\neq\alpha_i\in\Irr(S_i)$ that extends to $\Aut(S_i)$. Now, by Lemma 1.3 of \cite{mat}, $\beta_i=\alpha_i\times\cdots\times\alpha_i\in\Irr(N_i)$ extends to $\Aut(S_i)\wr \Sym(n_i)$. Now, it suffices to notice that $\beta=\beta_1\times\cdots\times\beta_t\in\Irr(\bF^*(G))$ is faithful and extends to $\Gamma$. 
		\end{proof}
		
	The next result, that appears as Corollary 10.5 of \cite{nav}, can be proved similarly. 
		
		\begin{thm}
Let $S$ be a nonabelian simple group. Let $N=S\times\cdots\times S$ be a minimal normal subgroup of a finite group $G$. Assume that $\varphi\in\Irr(S)$ extends to $I_{\Aut(S)}(\varphi)$. Put $\mu=\varphi\times\cdots\times\varphi\in\Irr(N)$. Then $\mu$ extends to $I_G(\mu)$. 
\end{thm}

Lemma 1.3 of \cite{mat} also allows us to control the field of values of the extended character. For instance, we have the following.

\begin{thm}[Lemma 2.2 of \cite{mproc}]
\label{3}
Let $\Gamma=G\wr A$ be a wreath product. Let $B=G\times\cdots\times G$ be the base group and let $\theta\in\Irr(B)$. Let $T=I_{\Gamma}(\theta)$. Then $\theta$ extends to $\eta\in\Irr(T)$ and $\eta^{\Gamma}\in\Irr(\Gamma)$. Furthermore, if $\chi=\eta^{\Gamma}$, then $\QQ(\chi)\subseteq\QQ(\eta)=\QQ(\theta)$.
\end{thm}

In short, the results in the previous section can be used in conjunction with Mattarei's lemma to get convenient  nonprincipal irreducible characters of nonabelian chief  factors of finite groups that extend irreducibly to $G$.

We conclude this section with a proof of Corollary B of \cite{mr}  using the ideas in this section instead of Theorem A of \cite{mr}. 

\begin{lem}
\label{sim}
Let $S$ be a nonabelian simple group and let $p$ be a prime number. Then there exists $\alpha\in\Irr_{p'}({\rm Aut}(S))$ faithful. 
\end{lem}

\begin{proof}
By Theorem D of \cite{ik}, there is $\alpha\in{\rm Irr}_{p'}({\rm Aut}(S))$ not having $S$ in its kernel. Since $S$ is the unique minimal normal subgroup of ${\rm Aut}(S)$, necessarily ${\rm Ker}(\alpha)=1$.
\end{proof}

\begin{thm}[Corollary C of \cite{mr}]
\label{rad}
 Let $G$ be a finite group with trivial solvable radical. Then there exists $\chi\in{\rm Irr}_{p'}(G)$ faithful.
\end{thm}

\begin{proof}
Write $F=\bF^*(G)$.
Since $G$ has trivial solvable radical we have that    $F$ is a direct product of nonabelian simple groups. Let $S_1,\dots, S_t$ be the pairwise nonisomorphic simple groups that appear as direct factors of  $F$ and write $F=N_1\times\cdots\times N_t$, with $N_i$ a direct product of, say $r_i$, copies of $S_i$ for every $i$. Notice that $\bC_G(F)\leq F$, so $\bC_G(F)=\bZ(F)=1$.  Thus $G$ is isomorphic to a subgroup of $\Gamma={\rm Aut}(F)={\rm Aut}(N_1)\times\cdots \times {\rm Aut}(N_t)$. 

%\textcolor{blue}{no pueden haber $N_i$, $N_j$ isomorfos? Ya no...antes si}

Then ${\rm Aut}(N_i)={\rm Aut}(S_i)\wr {\rm Sym(r_i)}$. Let $M={\rm Aut}(S_1)^{r_1}\times\cdots\times{\rm Aut}(S_t)^{r_t}$ and note that $M\lhd\Gamma$. 
By Lemma \ref{sim}, let $\beta_i\in{\rm Irr}_{p'}({\rm Aut}(S_i))$ faithful. Note that $\beta_i$ lies over  faithful characters of $S_i$. Let $\alpha_i$ be the product of $r_i$ copies of $\beta_i$, so $\alpha_i\in {\rm Irr}(\Aut(S_i)^{r_i})$ is $\Aut(N_i)$-invariant. By Lemma 1.3 of \cite{mat},  $\alpha_i$ extends to $\hat{\alpha}_i\in\Irr(\Aut(N_i))$.  Note that $\beta_i$ lies over a faithful character of $S_i$.

Put $\alpha=\alpha_1\times\alpha_2\times\cdots\times \alpha_t\in{\rm Irr}(M)$. Note that  $\alpha$ is also faithful and $\tau=\hat{\alpha}_1\times\cdots\times\hat{\alpha}_t$ is an extension of $\alpha$ to $\Gamma$.  Now, notice that $\tau$ has $p'$-degree and $\Ker\tau\cap F=1$. Since $F=\Soc(G)=\Soc(\Gamma)$ contains all the minimal normal subgroups of $\Gamma$, $\tau$ is faithful. Since  $\beta_i$ lies over a faithful character of $S_i$, it follows from the construction of $\tau$ that $\tau$ lies over faithful characters of $F$. Since $\tau$ has $p'$-degree, there exists $\chi\in\Irr_{p'}(G)$ lying under $\tau$. Notice that $\chi$ also lies over faithful characters of $F=\Soc(G)$. Hence, $\chi$ is faithful, as desired.
\end{proof}

\section{Simple groups and their automorphisms: regular orbits}

In Section 6 we have discussed the existence of non-principal irreducible characters of a non-abelian simple group $S$ that extend to $\Aut(S)$. Sometimes (but it seems that less often)  we are interested in the opposite, and want to find irreducible characters in $\Irr(S)$ in large $\Aut(S)$-orbits. It is not obvious that if $S<G\leq\Aut(S)$ then there exists $\varphi\in\Irr(S)$ that is not $G$-invariant. 

\begin{thm}
\label{inv}
Let $S$ be a non-abelian simple group.  If $g\in \Aut(S)-S$ then there exists $\varphi\in\Irr(S)$  such that $\varphi^g\neq\varphi$.
\end{thm}

\begin{proof}
By Theorem C of \cite{fs} (which relies on the classification of finite simple groups), there exists a conjugacy class of $S$ that is not fixed by $g$ (setwise). Now, apply Brauer's permutation lemma (Theorem 6.32 of \cite{isa}).
\end{proof}

It is well-known that regular orbits do not need to exist, even when $G/S$ is an abelian $p$-group (take $G=\Aut(\Sym(6))$) or when $G/S$ is cyclic (take $G=\Aut(\PSL_2(27))$).  As a consequence of Theorem \ref{inv}, we deduce that regular orbits exist when $G/S$ is cyclic of prime power order.

\begin{cor}
Let $S$ be a non-abelian simple group.  If $S<G\leq\Aut(S)$ and $G/S$ is a cyclic $p$-group for some prime $p$, then there exists $\varphi\in\Irr(S)$ such that $I_G(\varphi)=S$.
\end{cor}

\begin{proof}
Apply Theorem \ref{inv} to an outer automorphism in $G$ of order $p$.
\end{proof}

There are many applications of regular orbit theorems where more than one regular orbit is necessary. In the situation above, we have that if $p$ does not divide $|S|$, then there are at least two regular orbits.

\begin{thm}[Proposition 2.6 of \cite{mt}]
Let $S$ be a non-abelian simple group.  If $S<G\leq\Aut(S)$, $G/S$ is a cyclic $p$-group for some prime $p$ and $p$ does not divide $|S|$, then there exist $\varphi_1,\varphi_2\in\Irr(S)$ such that $I_G(\varphi_i)=S$ for $i=1,2$.
\end{thm}

In view of the lack of general results that guarantee the existence of large orbits, more technical results have been obtained. 

\begin{thm}[Theorem 2.3 of \cite{dnpst}]
Let $S$ be a non-abelian simple group.  If $S<G\leq\Aut(S)$ and $g\in G-S$ has odd order, then there exists $\varphi\in\Irr(S)$ such that $g$ does not fix any $G$-conjugate of $\varphi$.
\end{thm}

This result was important to obtain a version of the Isaacs-Navarro-Wolf conjecture on zeros of characters for arbitrary finite groups. As pointed out in \cite{dnpst}, $\Omega_8^+(2)$ shows that the hypothesis that the order of $g$ is odd is necessary. This group also shows that the actions of $G/S$ on $\Irr(S)$ and on the set of conjugacy classes of $S$ are not permutation isomorphic. A version of this result for $2$-elements has been recently obtained.

\begin{thm}[Corollary 2.2 of \cite{mt22}]
Let $S$ be a non-abelian simple group.  If $S<G\leq\Aut(S)$ and $gS\in \bF(G/S)$ is a nontrivial $2$-element, then there exists $\varphi\in\Irr(S)$ such that $g$ does not fix any $G$-conjugate of $\varphi$.
\end{thm}

\section{Nonsolvable groups: orbits on the power set}

As we have seen in Section 2, if $G$ is a nonsolvable group, then there exists $K\trianglelefteq G$ such that $G/K$ permutes the simple direct factors of the socle of $G/R(G)$. In this context, results on orbits of permutation groups on their power set are often useful. For instance, the following result was important to prove part (i) of Huppert's $\rho$-$\sigma$-conjecture (see Section 5). Recall that if a group $G$ acts on a set $\Omega$ and $\mu$ is a set of primes, we say that an orbit is $\mu$-semiregular if its size is divisible by all the primes in $\mu$.

\begin{thm}[Lemma 2.5 of \cite{mh}]
Suppose that a group $G$ acts on a finite set $\Omega$.  Let $n$ be the largest integer such that $A_n$ is a section of $G$.  Then $G$ has a  $\mu$-semiregular orbit on the power set $\PP(\Omega)$, where $\mu$ is the set of primes $p$ such that $p\geq\max\{32,(n+1)/2\}$. 
\end{thm}

Using the notation from Section 2, this theorem allows to find irreducible characters of $G$ whose degree is divisible by ``many" of the prime divisors of $|G/K|$ and thus constitutes an important step toward to a proof of part (i) of the $\rho$-$\sigma$-conjecture.  The bound obtained in \cite{mh} is quadratic. As mentioned at the end of  p. 3381 in \cite{mh}, the results in \cite{mh} reduce the problem of obtaining a linear bound for arbitrary finite groups to getting such a bound for groups $G$ such that $S_1\times\cdots\times S_t\leq G\leq \Aut(S_1)\times\cdots\times\Aut(S_t)$ for some nonabelian simple groups $S_1,\dots,S_t$. That is, it reduces the problem to groups ``without permutation part". 

This linear bound was achieved by C. Casolo and S. Dolfi in \cite{cd}.  Later, the Casolo-Dolfi bound was slightly improved to $|\rho(G)|\leq6\sigma(G)+1$.
Very recently, this bound has been improved to  $|\rho(G)|\leq\max\{5\sigma(G)+1,6\sigma(G)-4\}$ \cite{adp}. Modulo some conjecture in number theory, it was  shown in \cite{adp} (see also \cite{adps}) that  in a linear bound for the $\rho$-$\sigma$ problem for arbitrary groups, the coefficient that multiplies $\sigma(G)$ cannot be smaller than $3$.
 We think that the reason  why we are still far is that this type of results has always followed the structure presented in this paper: one considers first solvable groups, next groups with trivial solvable radical (or, in other words, works in $G/R(G)$) and finally glues the two pieces. 
We think that some approach that combines the two parts is necessary in order to get bounds that are close to best possible. We propose the following version of Question \ref{hupsol} for arbitrary finite groups. 

\begin{que}
\label{hupgen}
Let $G$ be a finite group. Do there exist $\chi_1,\chi_2, \chi_3\in\Irr(G)$ such that  $\rho(G)=\pi(\chi_1\chi_2\chi_3(1))$?
		\end{que}
		
		As the groups $\SL_2(2^n)$ for $n\geq 2$ or the Janko group $J_1$ show, less than $3$ irreducible characters do not suffice. The case of simple groups of  Question \ref{hupgen} was conjectured by Alvis (see \cite{ab}).  Actually, just the case of alternating groups remained open after \cite{ab}. This was settled in \cite{bw}. 

The following extension of Gluck's conjecture to nonsolvable groups was raised in \cite{chmn}: is it true that $|G:\bF(G)|\leq b(G)^3$ for any finite group $G$?  We propose the following strengthening of this question, along the lines of Questions \ref{glusol} and \ref{hupgen}.

\begin{que}
Let $G$ be a finite group. Do there exist $\chi_1,\chi_2,\chi_3\in\Irr(G)$ such that  $|G:\bF(G)|$ divides $\chi_1(1)\chi_2(1)\chi_3(1)$?
\end{que}

In this case, note that since $b(G)<|G|^{1/2}$ for any simple group $G$, any simple group shows that two characters are not enough.

%For instance, we propose the following question:

%\begin{con}
%Let $G$ be a finite group, let $Z=\bZ(G)$ be  cyclic and let  $\lambda\in\Irr(Z)$ be faithful.  Suppose that $G/Z$ has trivial solvable radical. Then there exists $\chi\in\Irr(G)$ lying over $\lambda$ such that $\chi(1)$ is divisible by at least one-third of the primes  in $\pi(G/Z)-\pi(Z)$?

%\begin{que}
%\label{rho}
%Let $R$ be the solvable radical of a finite group $G$. Is it true that there exist $\varphi\in\Irr(R)$ such that $|\pi(\varphi(1))|=\sigma(R)$  and $\chi\in\Irr(G)$ lying over $\varphi$ such that $\chi(1)$ is divisible by at least one-third of the primes in $\pi(G/R)-\pi(R)$?
%\end{que}

%The following immediate consequence shows the relevance of Question \ref{rho}.

%\begin{thm}
%Assume that Question \ref{rho} has an affirmative answer. Then $|\rho(G)|\leq3\sigma(G)$ for any finite group $G$.
%\end{thm}

The following theorem of L. Babai and L. Pyber \cite{bp} is another result on orbits on the power set that we have used to study another problem on characters, more precisely, Brauer's Problem 1.

\begin{thm}[Theorem 1 of \cite{bp}]
Let $G$ be a permutation group on a set $\Omega$ of cardinality $k$. Suppose that $G$ does not contain any alternating group bigger than $A_n$ as a composition factor. Then the number of orbits of $G$ on the power set $\PP(\Omega)$ is at least $a^{k/t}$, where $a>1$ is some constant. 
\end{thm}

We refer the reader to \cite{m07} where this result, in conjunction with many of the remaining ideas in this survey were used. The way to handle the case of solvable groups there follows different techniques than those in the large orbit Theorem \ref{lo} that we have considered here, but many of the ideas for nonsolvable groups that we have presented here were  first used there.

\section{Applications to problems that are not on character degrees}

So far we have focussed on problems on character degrees. As can be seen in \cite{mw} or \cite{mh}, for instance, the ideas we have presented here are also useful for problems on conjugacy class sizes. 
We conclude by describing some of the problems on character theory, that do not concern character degrees, but where these techniques are useful.

First, we recall the problem studied in \cite{m07} due to the formal similarities that it has with the problems that we will describe next. The goal in \cite{m07} was to prove that the order of a finite group $G$ can be bounded above in terms of the largest multiplicity of the character degrees. The case of $p$-groups had been proved in \cite{jz}. Building on this, we proved the case of solvable groups using \cite{fs} (this is another result that we have found useful many times) and then, with the techniques in this article we showed that the result holds for arbitrary finite groups if it holds for the symmetric groups. The case of symmetric groups was then solved by D. Craven \cite{cra}. 

Probably the problems where we have used more closely the ideas presented here concern zeros of characters. Given a group $G$, let $m^*(G)$ be the maximum number of zeros in a column of the character table of $G$. It was proved in Theorem A of \cite{ms1} that the number of nonlinear irreducible characters of $G$ is bounded above in terms of $m^*(G)$. (Notice that for any abelian group $G$, $m^*(G)=0$ so it is necessary to restrict ourselves to nonlinear characters).  As mentioned in the Introduction, using ad hoc techniques we solved this problem for $p$-groups (or, equivalently, nilpotent groups) in Section 3 of \cite{ms1}.  Using the large orbit Theorem \ref{lo}, we handled the case of solvable groups in Section 4.  Finally,  using the general picture described in Section 2 of this paper and Theorem \ref{extsim}, we proceeded with the general case in Section 5.

We have also considered a dual problem on zeros of characters. Given a group $G$, we write $m(G)$ to denote the maximum number of zeros in a row of the character table of $G$. Again, using ad hoc techniques it was proved in Theorem B of \cite{ms2} that if $P$ is a nonabelian $p$-group, then $|P|$ is bounded above in terms of $m(P)$. In Conjecture F of \cite{ms2} it was conjectured that for $G$ solvable, $|G:\bF(G)|$ is bounded above in terms of $m(G)$. As an immediate consequence of Theorem \ref{lo}, we proved in Theorem A of \cite{ms2} that $|G:\bF_8(G)|$ is bounded in terms of $m(G)$, but the conjecture for solvable groups remains open. (Actually, the result there is that $|G:\bF_{10}(G)|$ is bounded in terms of $m(G)$. We need Yang's improvement \cite{yan09} to get that $|G:\bF_8(G)|$ is bounded.)

Recently, we have considered this problem for arbitrary finite groups \cite{mjalg}.  In Theorem A of \cite{mjalg} we proved that for arbitrary finite groups $G$,  $|G:\bF_8(G)|$ is bounded above in terms of $m(G)$. We also conjectured the following (Conjecture 3.2 of \cite{mjalg}). The case of groups with trivial solvable radical (Theorem 2.3 of \cite{mjalg}), which was handled using the general picture from Section 2,  was an important part of the proof.  

\begin{con}
Let $G$ be a finite group. Then $|G:\bF(G)|$ is bounded in terms of $m(G)$. Furthermore, if $|\bF(G)|$ is not abelian, then $|G|$ is bounded in terms of $m(G)$. 
\end{con}

As can be seen from the proof of Theorem 3.3 of \cite{mjalg}, we expect the proof of this conjecture to rely on the $p$-group case Theorem B of \cite{ms2}. Some ingredient other than Theorem \ref{lo} seems necessary in order to get the solvable case of the conjecture. 

Finally, we mention two new areas where we have used these ideas recently: fields of values of characters and character kernels. We write
$$
f(G)=\max_{F/\QQ} |\{\chi\in\Irr(G)\mid\QQ(\chi)=F\}|,
$$ 
where $F/\QQ$ runs over the field extensions of $\QQ$. That is, $f(G)$ is the maximum multiplicity of the fields of values of the irreducible characters of $G$. We prove in Theorem A of \cite{mproc} that $|G|$ is bounded above in terms of $f(G)$ for every group $G$. The proof uses,  for instance, Theorem \ref{1}, \ref{2} and \ref{3}. 

On the other hand, in joint work with N. Rizo \cite{mr}, we have obtained an extension of the classical theorem of Broline-Garrison on character kernels (Theorem 12.19 of \cite{isa}) to 
irreducible characters of $p'$-degree, where $p$ is a prime.

Of course, all the results discussed in this section have formal similarities. However, the nature of character degrees, zeros of characters, fields of values, and character kernels are completely different. We think it is remarkable that these results have been obtained using similar techniques. In our opinion, this illustrates the power of the methods that have appeared in this paper.

%%%%%%%%%%%%%%%%%%%%%%%%%%%%%%%%%%%%%%%%%%%%%%%%%%%%%%%%%%%%%%%%%%%%%%%%%


\begin{thebibliography}{99} 

\bibitem{adp} Z. Akhlaghi, S. Dolfi, E. Pacifici, On Huppert's rho-sigma conjecture, J. Algebra {\bf  586} (2021), 537--560.

\bibitem{adps} Z. Akhlaghi, S. Dolfi, E. Pacifici, L. Sanus, Bounding the number of vertices in the degree graph of a finite group, 
J. Pure Appl. Algebra {\bf 224} (2020), 725-731.

\bibitem{ab} D. L. Alvis, M. Barry, Character degrees of simple groups, J. Algebra {\bf 140} (1991), 116--123.


\bibitem{bp}
L. Babai, L. Pyber, Permutation groups without exponentially many orbits on the power ser, J. Combin. Theory Ser. A {\bf 66} (1994), 160--168.

\bibitem{bw} M. Barry, M. Ward, On a conjecture of Alvis, J. Algebra {\bf 294} (2005), 136--155.


\bibitem{bclp}
M. Bianchi, D. Chillag, M. Lewis, E. Pacifici, Character degree graphs that are complete graphs, Proc. Amer. Math. Soc. {\bf 135} (2007), 671--676.


\bibitem{bw2}
T. Bonner, L. Wilson, On the Taketa bound for normally monomial
$p$-groups of maximal class II, J. Algebra Appl. {\bf 13} (2014), 19 pp.

\bibitem{bra}
R. Brauer, Representations of finite groups. In: Lectures on Modern Mathematics, Vol. I  Wiley, New York, (1963), 133-175.

\bibitem{cd}
C. Casolo, S. Dolfi, Prime divisors of irreducible character degrees and of conjugacy class sizes in finite groups, J. Group Theory {\bf 10} (2007), 571--583.

\bibitem{atl} J. H. Conway, R. T. Curtis, S. P. Norton, R. A. Parker, R. A. Wilson,``Atlas of Finite Groups", Clarendon Press, Oxford, 1985.



\bibitem{chmn}
J. P. Cossey, Z. Halasi, A. Mar\'oti, H. N. Nguyen, On a conjecture of Gluck, Math. Z. {\bf 279} (2015), 1067-1080.

\bibitem{cn}
J. P. Cossey, H. N. Nguyen, Controllling composition factors of a finite group by its character degree ratio, J. Algebra {\bf 403} (2014), 185--200.


\bibitem{cra}
D. Craven, Symmetric group character degrees and hook numbers, Proc. London Math. Soc. {\bf 96} (2008), 26--50.



\bibitem{dnpst} S. Dolfi, G. Navarro, E. Pacifici, L. Sanus, P. H. Tiep, Non-vanishing elements of finite groups, J. Algebra {\bf 323} (2010), 540--545.

\bibitem{em}
C. Eaton, A. Moret\'o, Extending Brauer's height zero conjecture to blocks with nonabelian defect group, Int. Res. Math. Not. (2014), 5581--5601. 

\bibitem{fs}
E. Farias e Soares, Big primes and character values for solvable groups, J. Algebra {\bf 100} (1986), 305--324.

\bibitem{fm}
G. Fern\'andez-Alcober, A. Moret\'o, Groups with two extreme character degrees and their normal subgroups, Trans. Amer. Math. Soc. {\bf 353} (2001), 2171--2192.


\bibitem{grs}
E. Giannelli, N. Rizo,  A. A. Schaeffer Fry, Groups with few
$p'$-character degrees, J. Pure Appl. Algebra {\bf 224} (2020), 15 pp. 

\bibitem{grss}
E. Giannelli, N. Rizo, B. Sambale, A. A. Schaeffer Fry, Groups with few
$p'$-character degrees in the principal block, Proc. Amer. Math. Soc.  {\bf 148} (2020), 4597-4614. 

\bibitem{hun}
N. N. Hung, Characters of $p'$-degree and Thompson's character degree theorem, Rev. Mat. Iberoam. {\bf 33} (2017), 117-138.

\bibitem{hls}
N. N. Hung, M. L. Lewis, A. A. Schaeffer Fry, Finite groups with an irreducible character of large degree, Manuscripta Math. {\bf 149} (2016), 523-546.

\bibitem{hstv}
N. N. Hung, A. A. Schaeffer Fry, H. P. Tong-Viet, C. Ryan Vinroot, On the number of irreducible real-valued characters of a finite group, J. Algebra {\bf 555} (2020), 275--288.

\bibitem{ht}
N. N. Hung, P. H. Tiep, The average character degree and an improvement of the It\^o-Michler theorem, J. Algebra {\bf 550} (2020), 86--107.

\bibitem{hup} B. Huppert, Endliche Gruppen I, Springer-Verlag, Berlin, Heidelberg, New York, 1967.

\bibitem{isa68}
M. Isaacs, Groups having at most three irreducible character degrees,
Proc. Amer.  Math. Soc.  {\bf 21} (1975),  146--151.


\bibitem{isa75}
M. Isaacs, Character degrees and derived length of a solvable group,
Canadian J. Math {\bf 27} (1969),  185--188.

\bibitem{isa}
M. Isaacs,
``Character Theory of Finite Groups",
Dover,
New York,
1994.

\bibitem{isa02}
M. Isaacs, Normal subgroups with nonabelian quotients in
$p$-groups, J. Algebra {\bf 247} (2002), 231--243.

\bibitem{ik}
M. Isaacs, G. Knutson, Irreducible character degrees and normal subgroups, J. Algebra {\bf 199} (1998), 302-326.

\bibitem{im}
M. Isaacs, A. Moret\'o, Character degrees and nilpotence class of a $p$-group, J. Algebra {\bf 238} (2001), 827--842.

\bibitem{inw}
M. Isaacs, G. Navarro, T. Wolf, Finite group elements where no irreducible character vanishes, J. Algebra {\bf 222} (1999), 413--423.

\bibitem{ip}
M. Isaacs, D. Passman, A characterization of groups in terms of the degrees of their characters. II,
Pacific J. Math {\bf 24} (1968), 467--510.  


\bibitem{is}
M. Isaacs, M. Slattery, Character degree sets that do not bound the class of a $p$-group, Proc. Amer. Math. Soc. {\bf 130} (2002), 2553--2558. 

\bibitem{jz}
A. Jaikin-Zapirain, On the number of conjugacy classes of finite $p$-group, J. London Math. Soc. {\bf 68} (2003), 699--711.

\bibitem{jk}
G. James, A. Kerber, ``The Representation Theory of the Symmetric Group", Addison-Wesley, Reading, Mass. 1981.


\bibitem{krt}
T. Keller, D. Ragan, G. T. Tims, On the Taketa bound for normally monomial
$p$-groups of maximal class, J. Algebra {\bf 277} (2004), 675--688.



\bibitem{ks}
H. Kurzweil, B. Stellmacher, ``The Theory of Finite Groups: An Introduction", Springer, New York, 2004.

\bibitem{ll}
Y. Liu, Z. Lu, A note on Huppert's $\rho$-$\sigma$ conjecture: An improvement on a result by Casolo and Dolfi, J. Algebra Appl. {\bf 13} (2014).

\bibitem{ly}
Y. Liu, Y. Yang, On Huppert's $\rho$-$\sigma$ conjecture, Monatsh. Math. {\bf 197} (2022), 299--309.

\bibitem{mal08}
G. Malle, Extensions of unipotent characters and the inductive McKay condition, J. Algebra {\bf 320} (2008), 2963--2980.


\bibitem{mm04}
G. Malle, A. Moret\'o, Nonsolvable groups with few character degrees, J. Algebra {\bf 294} (2005), 117--126.

\bibitem{mm07}
G. Malle, A. Moret\'o, A dual version of Huppert's $\rho$-$\sigma$ conjecture, Int. Math. Res. Not. (2007), 14pp. 

\bibitem{mawo} O. Manz, T. Wolf, ``Representations of Solvable Groups", Cambridge University Press, 1993. 

\bibitem{mat}
S. Mattarei,  On character tables of wreath products, J. Algebra {\bf 175} (1995), 157--178.

\bibitem{m03}
A. Moret\'o, Characters of $p$-groups and Sylow $p$-subgroups. In: Groups St. Andrews 2001 in Oxford, Vol II, 412--421, Cambridge University Press, 2003.

\bibitem{m05}
A. Moret\'o, Complex group algebras of finite groups: Brauer's problem 1, Electron. Res. Announc. Amer. Math. Soc. {\bf 11} (2005), 34--39.

\bibitem{mh}
A. Moret\'o, A proof of Huppert's $\rho$-$\sigma$
conjectures for nonsolvable groups, Int. Math. Res. Not. {\bf 54} (2005), 3375--3383.

\bibitem{m06}
A. Moret\'o,  A variation on theorems of Jordan and Gluck, J. Algebra {\bf 301} (2006),  274.-279.

\bibitem{m07}
A. Moret\'o, Complex group algebras of finite groups: Brauer's problem 1, Adv. Math. {\bf 208} (2007), 236--248.

\bibitem{mjalg}
A. Moret\'o, Landau's theorem and the number of conjugacy classes of zeros of characters, J. Algebra {\bf 577} (2021), 203-209.

\bibitem{mproc}
A. Moret\'o, Multiplicities of fields of values of irreducible characters of finite groups, Proc. Amer. Math. Soc. {\bf 149} (2021), 4109--4116.

\bibitem{mn}
A. Moret\'o, G. Navarro, Characters, exponents and defects in $p$-solvable groups, Israel J. Math., {\bf 249} (2022), 553--576.

\bibitem{mr}
A. Moret\'o, N. Rizo, Kernels of $p'$-degree irreducible characters, Mediterr J. Math. {\bf 19} (2022), 8pp.

\bibitem{ms1}
A. Moret\'o, J. Sangroniz, On the number of zeros in the columns of the character table of a group, J. Algebra {\bf 279} (2004), 726--736.

\bibitem{ms2}
A. Moret\'o, J. Sangroniz, On the number of conjugacy classes of zeros of characters, Israel J. Math. {\bf 142} (2004), 163--187.

\bibitem{mt}
A. Moret\'o, P. H. Tiep, Prime divisors of character degrees, J. Group Theory, {\bf 11} (2008), 341-356.

\bibitem{mt22}
A. Moret\'o, P. H. Tiep, Nonsolvable groups have a large proportion of vanishing elements, Israel J. Math., to appear.

\bibitem{mw}
A. Moret\'o, T. Wolf, Orbit sizes, character degrees and Sylow subgroups, Adv. Math. {\bf 184} (2004), 18--36.

\bibitem{nav10} 
G. Navarro, Problems in Character Theory. In: Character Theory of Finite Groups, Contemmp. Math. {\bf 524} (2010), 97--126.

\bibitem{nav2}
G. Navarro, Variations on the It\^o-Michler theorem on character degrees, Rocky Mountain J. Math. {\bf 46} (2016), 1363--1377.

\bibitem{nav}
G. Navarro
``Character Theory and the McKay Conjecture",
Cambridge Studies in Advanced Mathematics,
Cambridge,
2018.

\bibitem{nsur}
G. Navarro, Problems on characters: solvable groups, Publ. Mat., to appear.

\bibitem{nrs}
G. Navarro, N. Rizo, L. Sanus, Character degrees in separable groups, Proc. Amer. Math. Soc. {\bf 150} (2022), 2323--2329.

\bibitem{nt}
G. Navarro, P. H. Tiep, Degrees of rational characters of finite groups, Adv. Math. {\bf 224} (2010), 1121--1142.

\bibitem{rsv}
N. Rizo, A. A. Schaeffer Fry, C. Vallejo, Galois action on the principal block and cyclic Sylow subgroups, Algebra Number Theory {\bf 14} (2020), 1953-1979.

\bibitem{yan09}
Y. Yang, Orbits of the actions of finite solvable groups, J. Algebra {\bf 321} (2009), 2012--2021.



 \end{thebibliography}
 \end{document}